\documentclass[11pt]{amsart}
\usepackage{graphicx, amssymb, amsmath, amsthm}
\usepackage{epsfig}
\usepackage{hyperref}
\usepackage{mathrsfs}

\usepackage{tikz}
\usepackage{here}

\usepackage{pifont}

\usepackage{listings}
\numberwithin{equation}{section}
\usepackage{xcolor}

\newtheorem{theorem}{Theorem}[section]
\newtheorem{prop}[theorem]{Proposition}
\newtheorem{lemma}[theorem]{Lemma}

\newtheorem{corollary}[theorem]{Corollary}

\newtheorem{proposition}[theorem]{Proposition}
\newtheorem{example}[theorem]{Example}

\theoremstyle{definition}
\newtheorem{definition}[theorem]{Definition}
\newtheorem{remark}[theorem]{Remark}

\makeatletter
\newcommand{\Extend}[5]{\ext@arrow0099{\arrowfill@#1#2#3}{#4}{#5}}
\makeatother

\newcommand{\R}{\mathbb{R}}  
\newcommand{\Z}{\mathbb{Z}}

\newcommand{\CC}{\mathbb{C}}

\newcommand{\II}{\mathrm{i}}
\newcommand{\p}{\partial} 

\newcommand{\ud}{\mathrm{d}}

\newcommand{\IP}[2]{\langle #1, #2\rangle}
\newcommand{\RRe}[1]{\mathrm{Re}(#1)}
\newcommand{\IIm}[1]{\mathrm{Im}(#1)}
\newcommand{\bs}[1]{\boldsymbol{#1}}

\begin{document}
\title{A Construction of non-degenerate $\Z_{2}$-harmonic functions on $\R^{n}$}

\author[Dashen Yan]{Dashen Yan}
\address{Department of Mathematics, Stony Brook University, Stony Brook, NY 11794-3651 USA. E-mail: dashen.yan@stonybrook.edu}

\maketitle
\begin{abstract}
	We discover an explicit construction of non-degenerate $\Z_{2}$-harmonic functions on $\R^{n},n\geq 3$, using a variant of ellipsoidal coordinates on $\R^{n}$. The branching set of these examples is a codimension-$2$ ellipsoid, providing the first known family of non-degenerate $\Z_{2}$-harmonic $1$-forms on $\R^{n}$ with compact branching sets. Moreover, the graph of the related $\Z_{2}$-harmonic one form in $T^{*}\R^{n}$ can be obtained as a certain limit of a specific sequence of Lawlor's necks in $\CC^{n}=T^{*}\R^{n}$. 
\end{abstract}

\section{Introduction}\label{S_1}%%%%%%%%%%%%
	Let $(M^{n},g), n\geq 2$ be an oriented Riemannian manifold, and $\Sigma\subset M$ be an oriented codimension-$2$ submanifold. Let $L$ be a flat $\R$ bundle over $M\setminus\Sigma$, such that any small loop linking along $\Sigma$ has monodromy $-1$ and $(E,\nabla)$ be a vector bundle equipped with a connection. A section $\alpha\in \Gamma(M\setminus \Sigma, E\otimes L)$ is called a $\Z_{2}$-harmonic section if $\alpha,\nabla \alpha\in L_{loc}^{2}$ and satisfies the harmonic equation
		\begin{equation*}
			\nabla^{*}\nabla \alpha=0, \text{ on } M\setminus \Sigma,
		\end{equation*} 
		where the derivatives are taken with respect to the flat structure on $L$. Of particular interest are the cases where $E$ is a trivial line bundle with a trivial connection, a cotangent bundle with a Levi-Civita connection, or a spinor bundle with a spin connection when $M$ is spin. In these cases, the harmonic sections are referred to as $\Z_{2}$-harmonic functions, $1$-forms, and spinors, respectively. Alternatively, $\Z_2$-harmonic sections can be viewed as a section on a double branched cover, which is anti-invariant under the covering involution.\\
		
	Multi-valued harmonic sections play important roles in the compactification problems in gauge theory \cite{Taubes12,Haydys15,Mazzeo16,Esfahani24} and in the field of special holonomy and calibrated geometry \cite{Donaldson17A,SqiHe23}. Despite its importance, a general existence theorem for $\Z_{2}$-harmonic sections is impossible, as it involves the choice of $\Sigma$, which is intrinsically a nonlinear problem. Research in this direction has been largely example-driven. Several examples, both compact and non-compact, have been studied previously by many people \cite{Chen24,SqiHe24,SqiHe25,Haydiy23,Taubes20,Taubes24}.\\
    
    In this paper, we provide an explicit construction for non-degenerate $\Z_{2}$-harmonic functions when $M=\R^{n}, n\geq 3$ and $\Sigma$ are codimension-$2$ ellipsoids. Our motivation for the construction originates from Donaldson's program \cite{Donaldson17} on the adiabatic limit of $G_{2}$ manifolds with coassociative $K3$ fibrations and branched maximal sections, which are a nonlinear analog of non-degenerate $\Z_{2}$-harmonic $1$-forms, appearing as the potential adiabatic limiting description for the $G_{2}$ structure. The examples we construct serve as geometric models to resolve the problem when the branched maximal sections do not avoid the $(-2)$-classes. We leave the further discussion in this direction to Section 1 in \cite{Yan25}.\\
		
	To begin, we focus on $\Z_{2}$-harmonic functions. Let $N$ be the normal bundle of $\Sigma$, and use the inverse of the exponential to define the map $\zeta$ from a tubular neighborhood of $\Sigma$ to $N$ in $(M,g)$. Using co-orientation, we can regard $N$ as a complex line bundle over $\Sigma$. If $\sigma$ is a section of the dual normal bundle $N^{-1}$, we obtain a complex-valued function $\IP{\sigma}{\zeta}$, which we will simply denote as $\sigma \zeta$. The monodromy around $\Sigma$ defines a square root $N^{1/2}$ with dual $N^{-1/2}$. Similarly, we can define a complex-valued function
		\begin{equation*}
			\sigma_{p}\zeta^{p}, \sigma_{p}\in \Gamma(\Sigma, N^{-p}),
		\end{equation*}
		for any half-integer $p$. According to Donaldson \cite{Donaldson17}, any $\Z_{2}$-harmonic function $f$ admits asymptotic behavior near $\Sigma$
		\begin{equation*}
			f=\RRe{A\zeta^{1/2}}+O(|\zeta|^{3/2}), A\in \Gamma(\Sigma, N^{-1/2}),
		\end{equation*}
		and if $A\equiv 0$,
		\begin{equation*}
			f=\RRe{B\zeta^{3/2}}+O(|\zeta|^{5/2}), B\in \Gamma(\Sigma, N^{-3/2}).
		\end{equation*}
		Donaldson showed that if $A\equiv 0$ and $B$ are nowhere-vanishing, then $f$ admits a good deformation theory. Such $\Z_{2}$-harmonic functions are called \textbf{non-degenerate} and will appear in gluing construction, for example, \cite{SqiHe24, Yan25}.\\

	Here, we will only consider harmonic functions on $\R^{n}, n\geq 3$, with the branching set $\Sigma$ being a compact smooth embedded submanifold. Outside a large ball that contains $\Sigma$, $L$ is a trivial line bundle. In other words, we can pick a \textbf{single-valued branch} of the $\Z_{2}$-harmonic function. Moreover, different choices differ by a sign. In this setting, we can regard a $\Z_{2}$-harmonic function $f$ as a harmonic function outside a large ball. Suppose $f$ has polynomial growth, then classical theory implies:
		\begin{equation}\label{S_1_Eqn_Asm}
			f=P+O(|\bs{x}|^{2-n}),
		\end{equation} 
		where $P$ is a harmonic polynomial. Weifeng Sun \cite{Weifeng22} and later Haydys, Mazzeo, Takahashi \cite{Haydiy23} showed that when $\Sigma$ is fixed, there is a $1$-$1$ correspondence between $\Z_{2}$ harmonic functions and harmonic polynomials $P$ related by Equation \ref{S_1_Eqn_Asm}. In their existence result, little is known about the asymptotic behavior. On the other hand, recently, Donaldson found a family of $\Z_2$-harmonic functions on $\R^{3}$ with the branching set being codimension-$2$ ellipses using twistor theory \cite{Donaldson25}. The goal of this paper is to use a variant of ellipsoidal coordinates to give an explicit construction of non-degenerate $\Z_2$-harmonic functions on $\R^{n}$, which recovers Donaldson's construction when $n=3$.\\
	\begin{theorem}\label{S_1_Thm_Main}
		For any positive numbers $h_{1},\cdots, h_{n-1}$, there exists a non-degenerate $\Z_{2}$-harmonic function $f_{\bs{h}}$ on $\R^{n}$, whose branching set is a codimension-$2$ ellipsoid
			\begin{equation*}
				E_{\bs{h}}: \sum_{i=1}^{n-1}\frac{x_{i}^{2}}{h_{i}^{2}}=1, x_{n}=0,
			\end{equation*}
			and such that, at large distance, we can pick a single-valued branch of $f_{\bs{h}}$, on which
			\begin{equation*}
				f_{\bs{h}}=a_{0}-\sum_{i=1}^{n}a_{i}x_{i}^{2}+O(|\bs{x}|^{2-n}).
			\end{equation*}
			Here, let $S(y)=\prod_{i=1}^{n-1}(y+h_{i}^{2})$ and $a_{i}$ be constants given by
				\begin{equation*}
					\begin{split}
						&a_{i}=\frac{\prod_{j=1}^{n-1}h_{j}}{2}\int_{0}^{\infty}\frac{\ud u}{(u^{2}+h_{i}^{2})\sqrt{S(u^{2})}}, 1\leq i\leq n-1;\\
						&a_{n}=-\frac{\prod_{j=1}^{n-1}h_{j}}{2}\int_{0}^{\infty}\frac{S'(u^{2})\ud u}{S(u^{2})^{3/2}};\\
						&a_{0}=\frac{\prod_{j=1}^{n-1}h_{j}}{2}\int_{0}^{\infty}\frac{\ud u}{\sqrt{S(u^{2})}}.
					\end{split}
				\end{equation*}
	\end{theorem}
    
	\begin{prop}\label{S_1_Prop_Main}
		Let $P$ be a degree-two homogeneous harmonic polynomial with the origin as its only critical point, whose index is $n-1$. Let $C_{n}$ be a positive constant. Then there exists a unique non-degenerate $\Z_2$-harmonic function $f$ constructed in Theorem \ref{S_1_Thm_Main}, such that we can pick a single-valued branch of $f$, on which
			\begin{equation*}
				f=C_{n}+P+O(|\bs{x}|^{2-n}), |\bs{x}|\to \infty.
			\end{equation*}
		Furthermore, when $C_{n}$ goes to zero, the ellipsoid scales down to $0$, and $f$ converges to $P$ outside any ball that contains the ellipsoid. 
	\end{prop}
	\begin{remark}\label{S_1_Rmk_2D}
		The corresponding $\Z_2$-harmonic $1$-form $\ud f$ can be viewed as a higher-dimensional analog of $\RRe{\sqrt{z^{2}-\epsilon^{2}}\ud z}$ on $\CC$. Roughly speaking, this family of $1$-forms serves as a local model of a family of non-degenerate $\Z_{2}$-harmonic $1$-forms with a shrinking branching set. In \cite{Yan25}, we develop a gluing theorem that glues these examples to the regular zeros of harmonic $1$-forms on compact manifolds, thereby producing many new examples of $\Z_{2}$-harmonic $1$-forms on compact manifolds. 
	\end{remark}
	
	In a related direction, in the compact setting, a $\Z_{2}$-harmonic one form is the infinitesimal branched deformation of special Lagrangians \cite{SqiHe23}. In a suitable setting, if we regard $T^{*}\R^{n}$ as $\CC^{n}$, Lawlor's necks can be viewed as a $2$-valued graph. More precisely, there exists an $\Z_{2}$-potential function $F$ on $\R^{n}$, such that Lawlor's necks is the $2$-valued graph $\ud F$ from $\R^{n}$ to $T^{*}\R^{n}$. A version of the infinitesimal branched deformation in our case is the following.
		\begin{prop}
			Let $f_{\bs{h}}$ be constructed as above. Then there exists a family of $\Z_{2}$-potential functions $F_{t}$ that define Lawlor's necks and branch along $E_{\bs{h}}$, such that $tF_{t}$ converges to $f_{\bs{h}}$ in $C^{\infty}_{loc}(\R^{n}\setminus E_{\bs{h}})$.
		\end{prop}
		
\subsection*{Acknowledgement}
The author is grateful to his Ph.D advisor, Prof. LeBrun, as well as to Prof. Donaldson for their valuable discussions and encouragement. He would also like to thank Jingrui Cheng, Dylan Galt, Siqi He, and Yusen Xia for fruitful discussions and helpful comments. This work is partially supported by Simons Foundation.

\section{Ellipsoidal coordinates on $\R^{n}$}\label{S_2}%%%%%%%%%%%%
	Our construction is motivated by ellipsoidal harmonics in classical electromagnetism. These functions are solutions to the Laplacian equation on $\R^{n}$ under ellipsoidal coordinates, a coordinate system that is orthogonal and well-defined on an open dense subset in $\R^{n}$. The key to our construction is the discovery that a variant of the ellipsoidal coordinates can be extended to orthogonal coordinates on a double cover of $\R^{n}$ with the branching set being an ellipsoid.\\
	
	In this section, we will give a brief review of the ellipsoidal coordinates and their variant in $\R^{n}, n\geq 3$. Our main reference for this section is \cite{Dass12}. Consider a family of hypersurfaces $C_{i},1\leq i\leq n$ defined by
			\begin{equation}\label{S_2_Eqn_Linear}
				C_{i}:\frac{x_{1}^{2}}{b_{1}^{2}-\lambda_{i}}+\cdots +\frac{x_{n}^{2}}{b_{n}^{2}-\lambda_{i}}=1,
			\end{equation}
		where
			\begin{equation*}
				\lambda_{n}<b_{n}^{2}<\lambda_{n-1}<b_{n-1}^{2}<\cdots<\lambda_{1}<b_{1}^{2}.
			\end{equation*}
		Direct computation shows that the normal vector
			\begin{equation*}
				N_{i}=(\frac{2x_{1}}{b_{1}^{2}-\lambda_{i}},\cdots,\frac{2x_{n}}{b_{n}^{2}-\lambda_{i}}),
			\end{equation*}
		are orthogonal to each other. Solving the Equation \ref{S_2_Eqn_Linear} by induction on the dimension $n$, we conclude that
			\begin{equation}\label{S_2_Eqn_Coor_1}
				x_{i}^{2}=\frac{\prod_{j}(b_{i}^{2}-\lambda_{j})}{\prod_{k\neq i}(b_{i}^{2}-b_{k}^{2})}.
			\end{equation}
		The parameters $(\lambda_{i})$ define an orthogonal coordinate system in $\R^{n}$, which is known as ellipsoidal coordinates. In this coordinate, the Euclidean metric is 
			\begin{equation}\label{S_3_Eqn_Met_1}
				\sum_{i=1}^{n}\frac{\prod_{j\neq i}(\lambda_{j}-\lambda_{i})}{4\prod_{j}(b_{j}^{2}-\lambda_{i})}\ud \lambda_{i}^{2}.
			\end{equation}
		This follows from an algebraic lemma, which will be used repetitively throughout this paper.
			\begin{lemma}\label{S_2_Lemma_Com}
				Let $\sigma_{i,k}(\bs{y}), 1\leq i\leq n-1$ be the $k$-th symmetric polynomial of $n-2$ variables
					\begin{equation*}
						y_{1},\cdots, \hat{y_{i}}, \cdots, y_{n-1},
					\end{equation*}
				then 
					\begin{equation*}
						\sum_{i=1}^{n-1}\frac{(-1)^{n-2-k}y_{i}^{l}\sigma_{i,n-2-k}(\bs{y})}{\prod_{j\neq i}(y_{i}-y_{j})}=\delta_{lk}.
					\end{equation*}
			\end{lemma}
			\begin{proof}
				Let $V$ be the Vandermonde matrix
					\begin{equation*}
							\begin{bmatrix}
								1 & y_{1} &\cdots & y_{1}^{n-2}\\
								1 & y_{2} & \cdots & y_{2}^{n-2}\\
								\vdots &\vdots  &\ddots  & \vdots\\
								1 & y_{n-1}& \cdots & y_{n-1}^{n-2}
							\end{bmatrix},
					\end{equation*}
				and let $L$ be its inverse
					\begin{equation*}
							\begin{bmatrix}
								L_{0,1}&\cdots & L_{0,n-1}\\
								\vdots &\ddots & \vdots\\
								L_{n-2,1} &\cdots & L_{n-2,n-1}
							\end{bmatrix}.
					\end{equation*}
				Define $L_{i}(y)=\sum_{j=0}^{n-2}L_{j,i}y^{j}$, then the equality $V\cdot L=Id$ yields to
					\begin{equation*}
						L_{i}(y_{k})=
							\begin{cases}
								1, i=k\\
								0, i\neq k
							\end{cases}.
					\end{equation*}
				Lagrangian interpolation then implies
					\begin{equation*}
						L_{i}(y)=\frac{	f(y)}{(y-y_{i})f'(y_{i})},
					\end{equation*}
				where $f(y)=\prod_{j=1}^{n-1}(y-y_{j})$. As a consequence
					\begin{equation*}
						L_{k,i}=\frac{(-1)^{n-2-k}\sigma_{i,n-2-k}}{\prod_{j\neq i}(y_{i}-y_{j})}.
					\end{equation*}
				The lemma follows from the equality
					\begin{equation*}
						L\cdot V=Id,
					\end{equation*}
				then
					\begin{equation*}
						\sum_{i=1}^{n-1}\frac{(-1)^{n-2-k}\sigma_{i,n-2-k}y_{i}^{l}}{\prod_{j\neq i}(y_{i}-y_{j})}=\delta_{lk}.
					\end{equation*}
			\end{proof}
	Take the derivatives with respect to $\lambda_{j}$ we get
		\begin{equation*}
			\begin{split}
				\sum_{i=1}^{n}\frac{\p x_{i}}{\p \lambda_{j}}\cdot \frac{\p x_{i}}{\p \lambda_{j}}&=\sum_{i=1}^{n}\frac{\prod_{k\neq j}\prod_{l\neq j}(b_{i}^{2}-\lambda_{k})(b_{i}^{2}-\lambda_{l})}{4\prod_{k\neq i}(b_{i}^{2}-b_{k}^{2})\prod_{l}(b_{i}^{2}-\lambda_{l})}\\
				&=\sum_{i=1}^{n}\frac{\prod_{k\neq j}\prod_{l\neq i}(b_{i}^{2}-\lambda_{k})(b_{l}^{2}-\lambda_{j})}{4\prod_{k\neq i}(b_{i}^{2}-b_{k}^{2})\prod_{l}(b_{l}^{2}-\lambda_{j})}
			\end{split}
		\end{equation*}
		Expand the equation, we obtain
		\begin{equation*}
			\frac{1}{4\prod_{l}(b_{l}^{2}-\lambda_{j})} \sum_{i=1}^{n}\frac{\sum_{p=0}^{n-1}\sum_{q=0}^{n-1}(-1)^{n-1-p+q}b_{i}^{2p}\lambda_{j}^{q}\sigma_{j,n-1-p}(\bs{\lambda})\sigma_{i,n-1-q}(\bs{b})}{\prod_{k\neq i}(b_{i}^{2}-b_{k}^{2})}.
		\end{equation*}
		Here, $\sigma_{i,k}(\bs{\lambda}), \sigma_{i,k}(\bs{b})$ are the $k$-th symmetric polynomials of
			\begin{equation*}
				\begin{split}
					&\lambda_{1},\cdots, \hat{\lambda}_{i}, \cdots, \lambda_{n};\\
					&b_{1}^{2},\cdots, \hat{b_{i}^{2}}, \cdots, b_{n}^{2}.
				\end{split}
			\end{equation*}
		Apply Lemma \ref{S_2_Lemma_Com}, we obtain Equation \ref{S_2_Eqn_Coor_1}
			\begin{equation*}
				\sum_{i=1}^{n}\frac{\p x_{i}}{\p \lambda_{j}}\cdot \frac{\p x_{i}}{\p \lambda_{j}}=\frac{\sum_{p=0}^{n-1}(-\lambda_{j})^{p}\sigma_{j,n-1-p}(\bs{\lambda})}{4\prod_{l}(b_{l}^{2}-\lambda_{j})}=\frac{\prod_{l\neq j}(\lambda_{l}-\lambda_{j})}{4\prod_{l}(b_{l}^{2}-\lambda_{j})}.
			\end{equation*}
\subsection{Branched cover and modified ellipsoidal coordinate}\label{S_2_1}
	The key observation is that we can slightly modify the parameters to obtain orthogonal coordinates on a double cover of $\R^{n}$, where the branch locus is an ellipsoid in $\R^{n-1}$. To begin with, let
			\begin{equation*}
				\begin{split}
					&h_{i}^{2}=b_{i}^{2}-b_{n}^{2}, i<n;\\
					&\mu_{i}^{2}=\lambda_{i}-b_{n}^{2}, i<n;\\
					&\mu_{n}^{2}=b_{n}^{2}-\lambda_{n}.
				\end{split}
			\end{equation*}
		Here, $h_{i+1}^{2}<\mu_{i}^{2}<h_{i}^{2}, i<n-1$, and $0<\mu_{n-1}^{2}<h_{n-1}^{2},\mu_{n}^{2}>0$. After the change of parameters, Equation \ref{S_2_Eqn_Coor_1} becomes:
			\begin{equation}\label{S_2_Eqn_Coor_2}
				\begin{split}
					&x_{i}^{2}=\frac{(h_{i}^{2}+\mu_{n}^{2})\prod_{j<n}(h_{i}^{2}-\mu_{j}^{2})}{h_{i}^{2}\prod_{j\neq i}(h_{i}^{2}-h_{j}^{2})}, i<n,\\
					&x_{n}^{2}=\frac{\prod_{j}\mu_{j}^{2}}{\prod_{j<n}h_{j}^{2}}.
				\end{split}
			\end{equation}
	
	In this coordinate, the Riemannian metric becomes
		\begin{equation}\label{S_2_Eqn_Met_3}
			\frac{\prod_{j=1}^{n-1}(\mu_{j}^{2}+\mu_{n}^{2})}{\prod_{j=1}^{n-1}(h_{j}^{2}+\mu_{n}^{2})}\ud \mu_{n}^{2}+\sum_{i=1}^{n-1}\frac{(\mu_{n}^{2}+\mu_{i}^{2})\prod_{j\neq i,j\leq n-1}(\mu_{j}^{2}-\mu_{i}^{2})}{\prod_{j=1}^{n-1}(h_{j}^{2}-\mu_{i}^{2})}\ud \mu_{i}^{2}
		\end{equation}
	If we further let
		\begin{equation*}
			\mu_{i}^{2}=h_{i+1}^{2}+(h_{i}^{2}-h_{i+1}^{2})\sin(\theta_{i})^{2}, i<n-1,
		\end{equation*}
		the expression further reduces to
		\begin{equation}
		\begin{split}\label{S_2_Eqn_Coor_3}
			&x_{i}=\frac{\cos(\theta_{i})\sin(\theta_{i-1})\sqrt{(h_{i}^{2}+\mu_{n}^{2})(h_{i}^{2}-\mu_{n-1}^{2})}}{h_{i}\sqrt{\prod_{j<i-1}(h_{j}^{2}-h_{i}^{2})}\sqrt{\prod_{j>i}(h_{i}^{2}-h_{j}^{2})}}\sqrt{\prod_{j>i}(h_{i}^{2}-\mu_{j}^{2})}\sqrt{\prod_{j<i-1}(\mu_{j}^{2}-h_{j}^{2})},\\
					&x_{n}=\frac{\prod_{j}\mu_{j}}{\prod_{j<n}h_{j}}.
		\end{split}	
		\end{equation}
		\begin{remark}\label{S_2_Rmk_Coor}
			The above coordinates map can be extended to the cases where multiple $h_{i}$ are equal. 
		\end{remark}
		\begin{example}\label{S_2_Ex_Coor_1}
			When $n=3$, the equation \ref{S_2_Eqn_Coor_3} becomes
				\begin{equation*}
				\begin{split}
					&x_{1}=\frac{\cos(\theta_{1})\sqrt{(h_{1}^{2}+\mu_{3}^{2})(h_{1}^{2}-\mu_{2}^{2})}}{h_{1}};\\
					&x_{2}=\frac{\sin(\theta_{1})\sqrt{(h_{2}^{2}+\mu_{3}^{2})(h_{2}^{2}-\mu_{2}^{2})}}{h_{2}};\\
					&x_{3}=\frac{\mu_{3}\mu_{2}\sqrt{h_{2}^{2}+(h_{1}^{2}-h_{2}^{2})\cos(\theta_{1})^{2}}}{h_{1}h_{2}}.
				\end{split}
				\end{equation*}
			If we let $h_{1}\to h_{2}=1$ these coordinates degenerate to oblate-spherical coordinates
				\begin{equation*}
				\begin{split}
					&x_{1}=\cos(\theta_{1})\sqrt{(1+\mu_{3}^{2})(1-\mu_{2}^{2})};\\
					&x_{2}=\sin(\theta_{1})\sqrt{(1+\mu_{3}^{2})(1-\mu_{2}^{2})};\\
					&x_{3}=\mu_{3}\mu_{2}.
				\end{split}
				\end{equation*}
		\end{example}
	
	If we view $(\theta_{1},\cdots, \theta_{n-2})$ as spherical coordinates $\Theta$ that parametrize an ellipsoid, then the Equation \ref{S_2_Eqn_Coor_3} defines a smooth map
		\begin{equation*}
			\bs{x}(\Theta, \mu_{n-1}, \mu_{n}): S^{n-2}\times (-h_{n-1},h_{n-1})\times (-\infty,\infty)\to \R^{n}.
		\end{equation*}
		\begin{definition}\label{S_2_Defn_CoorMod}
			We call $\bs{x}(\Theta,\mu_{n-1},\mu_{n})$ a \textbf{modified ellipsoidal coordinate}, here the parameters $h_{i}$ are allowed to satisfy $0<h_{i+1}\leq h_{i}$.
		\end{definition}
	It is easy to see that $\bs{x}$ is generically a $2$-$1$ map, such that 
		\begin{equation*}
			\bs{x}(\Theta, \mu_{n-1}, \mu_{n})=\bs{x}(\Theta,-\mu_{n-1},-\mu_{n}).
		\end{equation*}
	Moreover, we will have the following proposition:
		\begin{prop}\label{S_2_Prop_Coor}
			The coordinates map $\bs{x}$ defines an orthogonal coordinate system on the double cover over $\R^{n}$, with a branching set
				\begin{equation*}
					E_{\bs{h}}: \frac{x_{1}^{2}}{h_{1}^{2}}+\cdots+\frac{x_{n-1}^{2}}{h_{n-1}^{2}}=1,
				\end{equation*}
			which corresponds to $\{\mu_{n}=\mu_{n-1}=0\}$. In particular, if we let
				\begin{equation*}
					z=\frac{\prod_{j=1}^{n-2}\mu_{j}}{2\prod_{i=1}^{n-1}h_{i}}\big(\mu_{n}^{2}-\mu_{n-1}^{2}+2\sqrt{-1}\mu_{n-1}\mu_{n}\big),
				\end{equation*}
			then $z=\zeta+O(r^{2})$, where $\zeta$ is the inverse of the exponential map from the tubular neighborhood to the normal bundle. 
		\end{prop}
	To begin, we first consider the case when $h_{i+1}<h_{i}$ and then derive a formula that can be extended to $h_{i+1}\leq h_{i}$. We define $\sigma_{k}(\bs{h}), \sigma_{j,k}(\bs{h}), \sigma_{k}(\bs{\mu})$ as the symmetric polynomials with respect to the multiples
		\begin{equation*}
		\begin{split}
			&h_{1}^{2},\cdots,h_{n-1}^{2},\\
			&\mu_{1}^{2},\cdots, \mu_{n-1}^{2},-\mu_{n}^{2}.
		\end{split}
		\end{equation*}
	Then we will have the following Lemma
		\begin{lemma}\label{S_2_Lemma_Coor_4}
			The symmetric polynomials $\sigma_{k}(\bs{\mu})$ are linear combinations of constants and $x_{i}^{2}$. In particular, we have
				\begin{equation}\label{S_2_Eqn_Coor_4}
				\begin{split}
					&\sigma_{k}(\bs{\mu})=\sigma_{k}(\bs{h})-\sigma_{k-1}(\bs{h})x_{n}^{2}-\sum_{i=1}^{n-1}\sigma_{i,k-1}(\bs{h})x_{i}^{2}, 1\leq k \leq n-1;\\
					&\sigma_{n}(\bs{\mu})=-\sigma_{n-1}(\bs{h})x_{n}^{2}.
				\end{split}
				\end{equation}
		\end{lemma}
		\begin{proof}
			Use the fact that
				\begin{equation*}
					\sum_{k=0}^{n-1}(-1)^{n-1-k}h_{i}^{2k}\sigma_{n-1-k}(\bs{h})=0,
				\end{equation*}
			then Equation \ref{S_2_Eqn_Coor_1} becomes
				\begin{equation*}
				\begin{split}
					&x_{i}^{2}=\frac{\sum_{k=0}^{n-2}(-1)^{n-1-k}h_{i}^{2k}\bigg(\sigma_{n-1-k}(\bs{\mu})-\sigma_{n-1-k}(\bs{h})-\frac{\sigma_{n-2-k}(\bs{h})\sigma_{n}(\bs{\mu})}{\sigma_{n-1}(\bs{h})}\bigg)}{\Pi_{j\neq i}(h_{i}^{2}-h_{j}^{2})},1\leq i\leq n-1\\
					&x_{n}^{2}=-\frac{\sigma_{n}(\bs{\mu})}{\sigma_{n-1}(\bs{h})}.
				\end{split}
				\end{equation*}
            Applying Lemma \ref{S_2_Lemma_Com}, we prove the lemma.
		\end{proof}

		\begin{proof}[Proof of Proposition \ref{S_2_Prop_Coor}]
			The above lemma implies
				\begin{equation}\label{S_2_Eqn_DefineE}
					\sum_{i=1}^{n-1}\frac{x_{i}^{2}}{h_{i}^{2}}=1-\frac{\sigma_{n-1}(\bs{\mu})}{\sigma_{n-1}(\bs{h})}+\frac{\sigma_{n-2}(\bs{h})\sigma_{n}(\bs{\mu})}{\sigma_{n-1}(\bs{h})^{2}}.
				\end{equation}
			When $\mu_{n}=\mu_{n-1}=0$, $\sigma_{n-1}(\bs{\mu})=\sigma_{n}(\bs{\mu})=0$. Therefore, the branching set for the coordinates map $\bs{x}(\Theta,\mu_{n-1},\mu_{n})$ is $E_{\bs{h}}$.\\
			
			The defining functions for $E_{\bs{h}}$ are $F=\sum_{i=1}^{n-1}\frac{x_{i}^{2}}{h_{i}^{2}}-1$ and $x_{n}$. It follows immediately that
				\begin{equation*}
					\frac{F}{|\nabla F|}+\sqrt{-1}x_{n}=\zeta+O(r^{2}),
				\end{equation*}
			where $r$ is a distance function from a small neighborhood of $E_{\bs{h}}$ to $E_{\bs{h}}$. The norm $|\nabla F|=2(\sum_{i=1}^{n-1}\frac{x_{i}^{2}}{h_{i}^{-4}})^{1/2}$ admits an expansion near $E_{\bs{h}}$
				\begin{equation*}
					|\nabla F|=\frac{2\prod_{j=1}^{n-2}\mu_{j}}{\prod_{i=1}^{n-1}h_{i}}+O(r)	
				\end{equation*}
			On the other hand, notice that $\mu_{n}^{2}+\mu_{n-1}^{2}$ is uniformly equivalent to $r$. Then the Equation \ref{S_2_Eqn_DefineE} implies
				\begin{equation*}
					\frac{F}{|\nabla F|}=\frac{\prod_{j=1}^{n-2}\mu_{j}}{2\prod_{i=1}^{n-1}h_{i}}\big(\mu_{n}^{2}-\mu_{n-1}^{2}\big)+O(r^{2}).
				\end{equation*}
			Therefore, $z$ is a coordinate in the normal direction, such that
				\begin{equation*}
					\zeta=z+O(r^{2}).
				\end{equation*}
			Here, $z^{1/2}$ can be written as
				\begin{equation*}
					\frac{\prod_{j=1}^{n-2}\sqrt{\mu_{j}}}{\sqrt{2}\prod_{i=1}^{n-1}\sqrt{h_{i}}}(\mu_{n}+\sqrt{-1}\mu_{n-1}).
				\end{equation*}
			In the end, this formula can be extended to the cases $h_{i+1}\leq h_{i}$.
			\end{proof}

\section{Construction}\label{S_3}%%%%%%%%%%%%
	Our strategy for the construction is to find solutions to the Laplacian equation, which takes the form
		\begin{equation*}
			f=f(\Theta,\mu_{n-1},\mu_{n}),
		\end{equation*}
		such that $f(\Theta,\mu_{n-1},\mu_{n})=-f(\Theta,-\mu_{n-1},-\mu_{n})$ and that
		\begin{equation*}
			f=\RRe{\tilde{B}(\Theta)(\mu_{n}+\sqrt{-1}\mu_{n-1})^{3}}+O(|z|^{5/2})
		\end{equation*}
		near $E_{\bs{h}}$. We will first derive a formula for the non-degenerate $\Z_{2}$-harmonic function when $h_{i+1}<h_{i}$ and extend it to the cases $h_{i+1}\leq h_{i}$. 
	
	\subsection{Separation of Variables}\label{S_3_1}%%%%%%%%
		We will use Equation \ref{S_2_Eqn_Met_3} to write down the Laplacian operator under $\mu_{i}$'s coordinates. Let $S(y)=\prod_{j=1}^{n-1}(h_{j}^{2}+y)$, then the Laplacian operator $\Delta$ is
			\begin{equation*}
				\frac{1}{\prod_{k=1}^{n-1}(\mu_{n}^{2}+\mu_{k}^{2})}\mathscr{L}_{n}-\sum_{j=1}^{n-1}\frac{1}{(\mu_{n}^{2}+\mu_{j}^{2})\prod_{k\neq j}(\mu_{k}^{2}-\mu_{j}^{2})}\mathscr{L}_{j},
			\end{equation*}
		where 
			\begin{equation*}
				\begin{split}
					&\mathscr{L}_{n}=S(\mu_{n}^{2})\frac{\p^{2}}{\p \mu_{n}^{2}}+\mu_{n}S'(\mu_{n}^{2})\frac{\p}{\p \mu_{n}};\\
					&\mathscr{L}_{j}=S(-\mu_{j}^{2})\frac{\p^{2}}{\p \mu_{j}^{2}}-\mu_{n}S'(-\mu_{j}^{2})\frac{\p}{\p \mu_{j}}.
				\end{split}
			\end{equation*}
		Let $Q(y)=\sum_{k=0}^{n-2} Q_{k}y^{k}$ be any polynomial with degree less than $n-1$ and $y_{j}=\mu_{j}^{2}, y_{n}=-\mu_{n}^{2}$. It follows from Lemma \ref{S_2_Lemma_Com} that
			\begin{equation*}
				\sum_{j=1}^{n}\frac{Q(y_{j})}{\prod_{k\neq j}(y_{j}-y_{k})}=0.
			\end{equation*}
		We further define that
			\begin{equation*}
				\begin{split}
					&\mathscr{L}_{n,Q}=\mathscr{L}_{n}-Q(-\mu_{n}^{2});\\
					&\mathscr{L}_{j,Q}=\mathscr{L}_{j}+Q(\mu_{j}^{2}).
				\end{split}
			\end{equation*}
		As a consequence, the Laplacian operator can also be expressed as 
			\begin{equation*}
				\frac{1}{\prod_{k=1}^{n-1}(\mu_{n}^{2}+\mu_{k}^{2})}\mathscr{L}_{n,Q}-\sum_{j=1}^{n-1}\frac{1}{(\mu_{n}^{2}+\mu_{j}^{2})\prod_{k\neq j}(\mu_{k}^{2}-\mu_{j}^{2})}\mathscr{L}_{j,Q}.
			\end{equation*}
		We separate the variables
			\begin{equation*}
				f=\prod_{j=1}^{n}f_{j}(\mu_{j}).
			\end{equation*}
		If $\mathscr{L}_{j,Q}f_{j}=0$, then $f$ is a solution to the Laplacian equation.
			\begin{remark}\label{S_3_Rmk_Wick}
				The equation $\mathscr{L}_{j,Q}f=0$ is identically the same as
					\begin{equation}\label{S_3_Eqn_Ell}
						S(-\mu^{2})\frac{\p^{2}}{\p \mu^{2}}f-\mu S'(-\mu^{2})\frac{\p}{\p \mu}f+Q(\mu^{2}).
					\end{equation}
				If we use Wick's rotation $\mu_{n}=\II \mu$ , then Equation \ref{S_3_Eqn_Ell} will become
					\begin{equation*}
						\mathscr{L}_{n,Q}f=0.
					\end{equation*}
			\end{remark}	
		We will begin by finding quadratic polynomial solutions to Equation \ref{S_3_Eqn_Ell}, which is closely related to harmonic polynomials with quadratic growth in $\R^{n}$. Assume $f=\mu^{2}-p$, then $f$ is a solution, which is equivalent to
			\begin{equation*}
				2S(-\mu^{2})-2\mu^{2}S'(-\mu^{2})+Q(\mu^{2})(\mu^{2}-p)=0.
			\end{equation*}
		Compare the coefficients
			\begin{equation*}
				\begin{split}
					&(M_{n-1}+Q_{n-2})\mu^{2(n-1)}=0;\\
					&(M_{j}+Q_{j-1}-pQ_{j})\mu^{2j}=0, 0<j<n-1;\\
					&M_{0}-pQ_{0}=0.
				\end{split}
			\end{equation*}
		Here $2S(-y)-2yS'(-y)=\sum_{j=0}^{n-1}M_{j}y^{j}$. It follows naturally that if $y=p$ is a solution to 
			\begin{equation}\label{S_3_Eqn_Ply}
				\sum_{j=0}^{n-1}M_{j}y^{j}=0,
			\end{equation}
		then we can solve $Q_{i},i=0,\cdots, n-2$ inductively. The equation $S(-y)=0$ has $n-1$ distinct positive roots, and henceforth $S(-y)-yS'(-y)=0$ has $n-1$ distinct positive solutions, which we denote as $p_{i}$. In particular, we can arrange $p_{i}$ such that
			\begin{equation*}
				0<p_{n-1}<h_{n-1}^{2}<\cdots<h_{2}^{2}<p_{1}<h_{1}^{2}.
			\end{equation*}
		It follows from construction that
			\begin{equation*}
				P_{i}=(\mu_{n}^{2}+p_{i})\prod_{j=1}^{n-1}(\mu_{j}^{2}-p_{i}),
			\end{equation*}		
			is a solution to the Laplacian equation, which is known as an ellipsoidal harmonic of degree $2$ \cite{Dass12}. The Lemma \ref{S_2_Lemma_Coor_4} implies that each $P_{i}$ is a harmonic polynomial of the form
				\begin{equation*}
					C_{0}+\sum_{i=1}^{n}C_{i}x_{i}^{2}.
				\end{equation*}
		
		On the other hand, given a polynomial solution to $\mathscr{L}_{Q}f=0$, we can produce another linearly independent solution in the following way.
			\begin{lemma}\label{S_3_Lemma_Second}
				If $P(\mu)$ is a non-zero polynomial solution of degree $k$ to Equation \ref{S_3_Eqn_Ell} then
					\begin{equation*}
						Q(\mu):=P(\mu)\int_{0}^{\mu}\frac{\ud u}{P^{2}(u)\sqrt{S(-u^{2})}}
					\end{equation*}
				is another solution to Equation \ref{S_3_Eqn_Ell}.
			\end{lemma}
			\begin{proof}
				It is easy to verify $Q(\mu)$ is a solution. To check $Q(\mu), P(\mu)$ are different solutions, let $\mu\to \infty$ in the imaginary axis. The integral
					\begin{equation*}
						\int_{0}^{\mu}\frac{\ud u}{P^{2}(u)\sqrt{S(-u^{2})}}=\int_{0}^{\II \infty}\frac{\ud u}{P^{2}(u)\sqrt{S(-u^{2})}}+O(|\mu|^{-(n-2+2k)}).
					\end{equation*}
				As a result, $Q(\mu)=AP(\mu)+O(|\mu|^{^{-(n-2+k)}})$, where $A$ is a constant, is not a polynomial, for the remaining term $O(|\mu|^{^{-(n-2+k)}})$ is non-zero.
			\end{proof}
		It should be remarked that when $P(\mu)$ is an even (odd) function, then $Q(\mu)$ is an odd (even) function. This observation will play a crucial role in the construction of the $\Z_{2}$-harmonic function.
	
	\subsection{Construction of non-degenerate $\Z_{2}$-harmonic function}\label{S_3_2}
		Follow from Proposition \ref{S_2_Prop_Coor} if $f=f(\Theta,\mu_{n-1},\mu_{n})$ is an odd function with respect to $(\mu_{n}, \mu_{n-1})$, then $f$ defines a multivalued function on $\R^{n}$ with monodromy $-1$ around the branching set $E_{\bs{h}}$. Recall that $S(y)=\prod_{j=1}^{n-1}(y+h_{j}^{2})$, and we define 
			\begin{equation}\label{S_3_Eqn_EllHar}
				\begin{split}
					&f_{0}=\int_{0}^{\mu_{n}}\frac{\ud u}{\sqrt{S(u^{2})}},\\
					&f_{2,i}=\bigg((\mu_{n}^{2}+p_{i})\int_{0}^{\mu_{n}}\frac{\ud u}{(u^{2}+p_{i})^{2}\sqrt{S(u^{2})}}\bigg)\prod_{j=1}^{n-1}(\mu_{j}^{2}-p_{i}).
				\end{split}
			\end{equation}
		Those harmonic functions are odd functions with respect to $(\mu_{n},\mu_{n-1})$, and hence, they are $\Z_{2}$-harmonic functions on $\R^{n}$. Moreover, $|f_{0}|$ is bounded, and $|f_{2,i}|$ grows no faster than quadratic in $\R^{n}$.\\
		
		The above $\Z_{2}$-harmonic functions are not yet non-degenerate. To obtain a non-degenerate $\Z_{2}$-harmonic function with quadratic growth, take Taylor's expansion at $(\mu_{n},\mu_{n-1})=(0,0)$. Let $z$ be the complex coordinates defined in Proposition \ref{S_2_Prop_Coor}, we will have
			\begin{equation*}
				\begin{split}
					&f_{0}=\frac{1}{\prod_{j=1}^{n-1}h_{j}}\bigg(\mu_{n}-\frac{1}{6}\big(\sum_{j}^{n-1}\frac{1}{h_{j}^{2}}\big)\mu_{n}^{3}\bigg)+O(|z|^{\frac{5}{2}});\\
					&f_{2,i}=\frac{\prod_{k=1}^{n-2}(\mu_{j}^{2}-p_{i})}{\prod_{j=1}^{n-1}h_{j}}\bigg(\frac{1}{6}\big(-\frac{2}{p_{i}}+\sum_{j=1}^{n-1}\frac{1}{h_{j}^{2}}\big)\mu_{n}^{3}-\mu_{n}+\frac{1}{p_{i}}\mu_{n}\mu_{n-1}^{2}\bigg)+O(|z|^{\frac{5}{2}}).
				\end{split}
			\end{equation*}
		Apply Lemma \ref{S_2_Lemma_Com}, we define the $\Z_{2}$-harmonic function $f_{\bs{h}}$ in Theorem \ref{S_1_Prop_Main}
			\begin{equation*}
				f_{\bs{h}}:=\frac{1}{n}\sqrt{\sigma_{n-1}(\bs{h})}\bigg(f_{0}+(-1)^{n-2}\sum_{i=1}^{n-1}\frac{1}{\prod_{i\neq j}(p_{i}-p_{j})}f_{2,i}\bigg).
			\end{equation*}
		The asymptotic behavior of $f_{\bs{h}}$ at $(\mu_{n},\mu_{n-1})=(0,0)$ is
			\begin{equation*}
				f_{\bs{h}}=\frac{(-1)^{n-2}}{n}\big(\prod_{k=1}^{n-2}\mu_{j}^{2}\big)\bigg(\sum_{i=1}^{n-1}\frac{1}{3p_{i}\prod_{k\neq j}(p_{i}-p_{k})}\bigg)\bigg(-\mu_{n}^{3}+3\mu_{n}\mu_{n-1}^{2}\bigg)+O(|z|^{\frac{5}{2}}).
			\end{equation*}
		Notice that $p_{i}$ is a solution to
			\begin{equation*}
				S(-y)-yS'(-y)=0,
			\end{equation*}
		and henceforth
			\begin{equation*}
				(-1)^{n-2}\frac{1}{p_{i}}=\frac{np_{i}^{n-2}}{\prod_{j=1}^{n-1}h_{j}^{2}}+l.o.t.
			\end{equation*}
		We now conclude
			\begin{equation*}
				f_{\bs{h}}=-\frac{\prod_{k=1}^{n-2}\mu_{j}^{2}}{3\prod_{j=1}^{n-1}h_{j}^{2}}(\mu_{n}^{3}-3\mu_{n} \mu_{n-1}^{2})+O(|z|^{\frac{5}{2}}).
			\end{equation*}
		In terms of the complex coordinate
			\begin{equation*}
				f_{\bs{h}}=-\frac{2^{3/2}}{3}\frac{\prod_{k=1}^{n-2}\sqrt{\mu_{j}}}{\prod_{j=1}^{n-1}\sqrt{h_{j}}}\RRe{z^{\frac{3}{2}}}+O(|z|^{\frac{5}{2}})
			\end{equation*}
		The coefficient 
			\begin{equation*}
				B(\Theta):=-\frac{2^{3/2}}{3}\frac{\prod_{k=1}^{n-2}\sqrt{\mu_{j}}}{\prod_{j=1}^{n-1}\sqrt{h_{j}}}
			\end{equation*}
		is strictly less than $0$ on the branching set $E_{\bs{h}}$.\\
	
\section{Asymptotic behavior of the $\Z_{2}$-harmonic functions}		
		We now study the asymptotic behavior of $f_{\bs{h}}$ as $|\bs{x}|\to \infty$. As commented in Lemma \ref{S_3_Lemma_Second}, define
			\begin{equation*}
				\begin{split}
					&A_{0}=\int_{0}^{\infty}\frac{\sqrt{\sigma_{n-1}(\bs{h})}\ud u}{\sqrt{S(u^{2})}};\\
					&A_{2,i}=\int_{0}^{\infty}\frac{\sqrt{\sigma_{n-1}(\bs{h})}\ud u}{(u^{2}+p_{i})^{2}\sqrt{S(u^{2})}}.
				\end{split}
			\end{equation*}
		On a single-valued branch, $f_{\bs{h}}$ satisfies
			\begin{equation*}
				f_{\bs{h}}-\bigg(A_{0}+(-1)^{n-2}\sum_{i=1}^{n-1}\frac{A_{2,i}P_{i}}{\prod_{k\neq i}(p_{i}-p_{k})}\bigg)=O(|\bs{x}|^{2-n}).
			\end{equation*}
		The main proposition in this section is as follows:
			\begin{prop}\label{S_4_Prop_Asym}
				There exist parameters $a_{i}, 1\leq i \leq n$, depending on $\bs{h}$, such that
					\begin{equation*}
						f_{\bs{h}}=a_{0}-(\sum_{j=1}^{n}a_{j}x_{j}^{2})+O(|\bs{x}|^{2-n}),
					\end{equation*}
				where 
					\begin{equation*}
					\begin{split}
						&a_{j}=\frac{1}{2}\int_{0}^{\infty}\frac{\sqrt{\sigma_{n-1}(\bs{h})}\ud u}{(u^{2}+h_{j}^{2})\sqrt{S(u^{2})}}, 1\leq i\leq n-1;\\
						&a_{n}=-\sum_{j=1}^{n-1}a_{j};\\
						& a_{0}=\frac{1}{2}\int_{0}^{\infty}\frac{\sqrt{\sigma_{n-1}(\bs{h})}\ud u}{\sqrt{S(u^{2})}}.
					\end{split}
					\end{equation*}
				In other words, $f_{\bs{h}}-C$ is asymptotic to a harmonic quadric of index $n-1$.
			\end{prop}
				
			\begin{proof}
				Expand $P_{i}$
					\begin{equation*}
						P_{i}=(-1)^{n-1}\sum_{k=0}^{n}(-1)^{n-k}p_{i}^{k}\sigma_{n-k}(\bs{\mu}).
					\end{equation*}
				Notice that
					\begin{equation*}
						\sum_{k=0}^{n-1}(-1)^{n-1-k}p_{i}^{k}\sigma_{n-1-k}(\bs{p})=0,
					\end{equation*}
				then
					\begin{equation*}
					\begin{split}
						(-1)^{n-2}P_{i}&=\bigg(\sum_{k=1}^{n-2}(-1)^{n-k}\sigma_{n-k}(\bs{p})p_{i}^{k}\bigg)\\
							&+\bigg(\sum_{k=0}^{n-2}(-1)^{n-1-k}p_{i}^{k}(\sigma_{n-k}(\bs{\mu})+(\sigma_{1}(\bs{p})-\sigma_{1}(\bs{\mu}))\sigma_{n-1-k}(\bs{p}))\bigg)
					\end{split}
					\end{equation*}
				Using the Lemma \ref{S_2_Lemma_Coor_4}, we deduce
					\begin{equation}\label{S_4_Eqn_CooA}
						a_{j}=\frac{\sqrt{\sigma_{n-1}(\bs{h})}}{n}\sum_{i=1}^{n-1}\sum_{k=0}^{n-2}\int_{0}^{\infty}\frac{(-1)^{n-1+k}(\sigma_{n-k-1,j}(\bs{h})-\sigma_{n-1-l}(\bs{p}))p_{i}^{k}\ud u}{\prod_{l\neq i}(p_{i}-p_{l})(u^{2}+p_{j})^{2}\sqrt{S(u^{2})}}.
					\end{equation}
				We apply Lemma \ref{S_2_Lemma_Com} and obtain the following equality
					\begin{equation*}
					\begin{split}
						\sum_{i=1}^{n-1}\frac{p_{i}^{k}}{(y+p_{i})^{2}\prod_{l\neq i}(p_{i}-p_{l})}&=-\frac{\ud }{\ud y}\bigg(\sum_{i=1}^{n-1}\frac{p_{i}^{k}}{(y+p_{i})\prod_{l\neq i}(p_{i}-p_{l})}\bigg)\\
							&=-\frac{\ud }{\ud y}\bigg(\sum_{i=1}^{n-1}\frac{\sum_{q=0}^{n-2}y^{q}\sigma_{i,n-2-q}(\bs{p})p_{i}^{k}}{\prod^{n-1}_{l=1}\prod_{q\neq i}(y+p_{l})(p_{i}-p_{q})}\bigg)\\
							&=-\frac{\ud }{\ud y}\bigg(\frac{(-1)^{n-2+k}y^{k}}{\prod_{l=1}^{n-1}(y+p_{l})}\bigg), 0\leq k\leq n-2.
					\end{split}
					\end{equation*}
				Then the Equation \ref{S_4_Eqn_CooA} reduces to
					\begin{equation*}
						a_{j}=\frac{1}{n}\int_{0}^{\infty}\frac{\sqrt{\sigma_{n-1}(\bs{h})}}{2u\sqrt{S(u^{2})}}\frac{\ud }{\ud u}\bigg(\frac{u^{2}S(u^{2})}{(u^{2}+h_{j}^{2})\prod_{l=1}^{n-1}(u^{2}+p_{l})}\bigg) \ud u.
					\end{equation*}
				Notice that $\frac{\ud }{\ud y}yS(y)=n\prod_{l=1}^{n-1}(y+p_{l})$ and henceforth, integration by parts yields to
					\begin{equation*}
						a_{j}=\frac{1}{2}\int_{0}^{\infty}\frac{\sqrt{\sigma_{n-1}(\bs{h})}\ud u}{(u^{2}+h_{j}^{2})\sqrt{S(u^{2})}}
					\end{equation*}
                Similarly, using Equation \ref{S_2_Eqn_Coor_4}, we conclude
					\begin{equation*}
						a_{0}=\frac{1}{2}\int_{0}^{\infty}\frac{\sqrt{\sigma_{n-1}(\bs{h})}\ud u}{\sqrt{S(u^{2})}}.
					\end{equation*}
			\end{proof}
			\begin{remark}
				If we let $t^{-1}\bs{h}=(t^{-1}h_{i})$, then the $\Z_{2}$-harmonic function $t^{-1}f_{t^{-1}\bs{h}}$ has the same quadric term as $f_{\bs{h}}$ in the asymptotic expansion, but the constant term becomes $t^{-1}a_{0}$, which will go to $0$ as $t\to \infty$. Moreover, in this case, the ellipsoid scales down to the origin.
			\end{remark}

			To prove the Proposition \ref{S_1_Prop_Main}, we will need to show that any multiple of $n-1$ positive parameters can be reached by $(a_{1},\cdots, a_{n-1})$. To see this, define a cone $K_{n}$ in $\R^{n-1}$ as follows
				\begin{equation*}
					K_{n}:=\{(y_{1},\cdots, y_{n-1})|y_{1}, \cdots, y_{n-1}> 0\},
				\end{equation*}
			and a map 
				\begin{equation*}
					F_{n}:K_{n}\to K_{n},(\frac{1}{h_{1}},\cdots,\frac{1}{h_{n-1}})\mapsto (a_{1},\cdots, a_{n-1}).
				\end{equation*}
				\begin{corollary}
					The map $F_{n}$ is surjective. In other words, for every index $n-1$ harmonic quadric, there exists a non-degenerate $\Z_{2}$-harmonic polynomial that asymptotes to it.
				\end{corollary}
				\begin{proof}
					It is easy to see that $F_{n}=(F_{n,1},\cdots, F_{n,n-1})$ can be extended continuously to the closure $\overline{K_{n}}$. Indeed, $\p \overline{K_{n}}$ is a union of $\overline{K_{n-1}}$ and $F_{n}$ that can be continuously extended to $F_{n-1}$ on $\overline{K_{n-1}}$. Let $P_{n}$ be the radial projection from $\overline{K_{n}}\setminus\{0\}$ to the intersection
						\begin{equation*}
							\Delta_{n}=\overline{K_{n}}\cap \{y_{1}+\cdots + y_{n}=1\}.
						\end{equation*}
					Moreover, $F_{n}$ satisfies
						\begin{equation*}
							F_{n}(ty_{1},\cdots,ty_{n-1})=tF_{n}(y_{1},\cdots,y_{n-1}),
						\end{equation*}
					then $F_{n}$ is surjective if and only if $P_{n}\circ F_{n}: \Delta_{n}\to \Delta_{n}$ is surjective. This amounts to showing that $P_{n}\circ F_{n}|_{\p \Delta_{n}}$ is homotopic to $Id|_{\p \Delta_{n}}$, where the homotopy is
						\begin{equation*}
							H_{n}(\bs{x},t)=t Id+(1-t)P_{n}\circ F_{n}.
						\end{equation*} 
					When $n=3$, surjectivity is reduced to the intermediate value theorem. On the other hand, when $n>3$, we prove by contradiction and induction. Suppose $P_{n}\circ F_{n}$ is not surjective. By the induction hypothesis, let $p$ be a point in the interior of $\Delta_{n-1}$, which does not lie in the image of $P_{n} \circ F_{n}$. Now let $\pi_{p}$ be the radial projection from $\Delta_{n}\setminus\{p\}$ to $\p \Delta_{n}$. Therefore,
						\begin{equation*}
							P_{n}\circ F_{n}|_{\p \Delta_{n}}=\pi_{p}\circ P_{n}\circ F_{n}|_{\p \Delta_{n}}
						\end{equation*}
						is homotopy to a constant map. This contradicts the fact that $\p \Delta_{n}$ is homeomorphic to $S^{n-2}$.
				\end{proof}
			\begin{remark}
				The injectivity of the map $F_{n}$ was recently shown by Yuanbo Zhou through direct computation \cite{Commu_1}. In fact, to verify the injectivity, it suffices to show that the Jacobian $DF_{n}$ is injective on $K_{n}$.  Let $\Phi=a_{0}\prod_{j=1}^{n-1} h_{i}^{-1}$, Yuanbo Zhou observed that
                    \begin{equation*}
                        -\frac{1}{2}a_{i}=\frac{\p }{\p h_{i}}\Phi.
                    \end{equation*}
                The injectivity of $DF_{n}$ then follows immediately from the fact that the Hessian for the integrand of $\Phi$ is strictly positive with respect to the variables $h_{1},\cdots, h_{n-1}$.
			\end{remark}

\section{Relation to Donaldson's Twistor construction}\label{S_5}%%%%%%%%
	In the recent paper \cite{Donaldson25}, Simon Donaldson discovered another construction of $\Z_{2}$-harmonic functions using the twistor method when $n=3$. In this section, we demonstrate that our example coincides with Donaldson's when $n=3$. In his construction, the $\Z_{2}$-harmonic function is represented by a contour integral in twistor space, which can be written as
		\begin{equation}\label{S_5_Eqn_D}
			f=\RRe{\int_{0}^{2\pi}Q^{-3/2}\tan^{-1}(\frac{w}{\sqrt{Q}})(w^{2}+Q)\ud \theta}-\int_{0}^{2\pi}Q^{-1}x_{3}\ud \theta.
		\end{equation}
	Here,
		\begin{equation*}
		\begin{split}
			&Q=1+\epsilon \cos(2\theta), -1<\epsilon<1;\\
			&w=x_{3}+\sqrt{-1}\big(x_{1}\cos(\theta)+x_{2}\sin(\theta)\big),
		\end{split}
		\end{equation*}
	The branching set of $f$ is an ellipse in the $x_{1}x_{2}$-plane
		\begin{equation*}
			\frac{x_{1}^{2}}{1+\epsilon}+\frac{x_{2}^{2}}{1-\epsilon}=1.
		\end{equation*}

	Suppose $x_{3}\geq 0$, we can choose a single-valued branch of $\tan^{-1}(\frac{w}{\sqrt{Q}})$. Although the function $\tan^{-1}(\frac{w}{\sqrt{Q}})$ will develop logarithmic singularities at $w=\pm \sqrt{-1}\sqrt{Q}$, the integrand remains bounded.\\
	
	The main ingredient for the proof is the uniqueness theorem established by Weifeng Sun \cite{Weifeng22}.
		\begin{prop}[Theorem 0.4]\label{S_5_Prop_Uniqueness}
			If $K$ is a compact smoothly embedded oriented codimension-$2$ submanifold in $\R^{n}, n\geq 3$. Then there is a $1$-$1$ correspondence (up to a sign) between a $\Z_{2}$-harmonic function $f$ that branches along $K$ with polynomial growth and a harmonic polynomial $P$. The bijection is given by
				\begin{equation*}
					f-P\to 0, |\bs{x}|\to \infty.
				\end{equation*} 
			In particular, the difference is bounded by $O(|\bs{x}|^{2-n})$.
		\end{prop}
	
	To show that two constructions agree when $n=3$, it suffices to prove that the asymptotic behaviors of the $\Z_{2}$-harmonic functions at large distances agree. We now estimate Donaldson's examples as in \cite{Donaldson25}. Decompose $S^{1}$ into $I_{1}\sqcup I_{2}$
		\begin{equation*}
			I_{1}=\{|w|<|\bs{x}|^{\delta}\}, I_{2}=S^{1}\setminus I_{1},
		\end{equation*}
	where $0<\delta\ll 1$. Using Taylor's expansion, we conclude that, on $I_{2}$
		\begin{equation*}
			\tan^{-1}(\frac{w}{\sqrt{Q}})=\frac{\pi}{2}+\frac{\sqrt{Q}}{w}+O(|w|^{-3}).
		\end{equation*}
	On the other hand, the measure of $I_{1}$ is bounded by $C|\bs{x}|^{\delta-1}$ and the integrand is bounded by $C|\bs{x}|^{2\delta}$. The $\Z_{2}$-harmonic function satisfies
		\begin{equation*}
			f=\RRe{\int_{0}^{2\pi}Q^{-3/2}\frac{\pi}{2}(w^{2}+Q)\ud \theta}+O(\max(|\bs{x}|^{3\delta-1},|\bs{x}|^{-\delta})).
		\end{equation*}
	Direct computation shows that
		\begin{equation*}
			\RRe{\int_{0}^{2\pi}Q^{-3/2}\big(\frac{\pi}{2}(w^{2}+Q)\big)\ud \theta}=c_{0}-c_{1}x_{1}^{2}-c_{2}x_{2}^{2}+c_{3}x_{3}^{2},
		\end{equation*}
	where
		\begin{equation*}
			\begin{split}
				&c_{0}=\frac{\pi}{2}\int_{0}^{2\pi}\frac{\ud \theta}{(1+\epsilon\cos(2\theta))^{1/2}};\\
				&c_{1}=\frac{\pi}{2}\int_{0}^{2\pi}\frac{\cos(\theta)^{2}\ud \theta}{(1+\epsilon\cos(2\theta))^{1/2}};\\
				&c_{2}=\frac{\pi}{2}\int_{0}^{2\pi}\frac{\sin(\theta)^{2}\ud \theta}{(1+\epsilon\cos(2\theta))^{1/2}};\\
				&c_{3}=c_{1}+c_{2}.
			\end{split}
		\end{equation*}
	Changing the variable in the integrand yields
		\begin{equation*}
			\begin{split}
				&c_{0}=2\pi\int_{0}^{\infty}\frac{\ud u}{\sqrt{(1-\epsilon+u^{2})(1+\epsilon+u^{2})}};\\
				&c_{1}=2\pi\int_{0}^{\infty}\frac{\ud u}{(1+\epsilon+u^{2})\sqrt{(1-\epsilon+u^{2})(1+\epsilon+u^{2})}};\\
				&c_{2}=2\pi\int_{0}^{\infty}\frac{\ud u}{(1-\epsilon+u^{2})\sqrt{(1-\epsilon+u^{2})(1+\epsilon+u^{2})}}.
			\end{split}
		\end{equation*}
	If we let $h_{1}=\sqrt{1+\epsilon}, h_{2}=\sqrt{1-\epsilon}$, then the ratios among $a_{i}$ constructed as in Proposition \ref{S_4_Prop_Asym} are identical to those among $c_{i}$.

\section{Lawlor's necks with small angle}\label{S_6}%%%%%%%
	In a slightly different vein, the differential $\ud f_{\bs{h}}$ defines a $\Z_{2}$-harmonic one form on $\R^{n}$. Regard $T^{*}\R^{n}=(x_{1},y_{1}\frac{\p}{\p x_{1}}, \cdots, x_{n}, y_{n}\frac{\p}{\p x_{n}})$ as $\CC^{n}$ such that $z_{i}=x_{i}+\sqrt{-1}y_{i}$, and let $(\CC^{n},\omega,g,\Omega)$ be the flat Calabi-Yau structure, where $g$ is the Euclidean metric on $\CC^{n}$ and
		\begin{equation*}
			\omega=\sum_{i=1}^{n}\frac{\sqrt{-1}}{2}\ud z_{i}\wedge \ud \overline{z_{i}}, \Omega=\ud z_{1}\wedge \cdots\wedge z_{n}.
		\end{equation*}	
	Moreover, the symplectic structure $\omega=\ud \lambda$ is exact with the Liouville form $\lambda=-\sum_{i=1}^{n} y_{i}\ud x_{i}$. A \textbf{special Lagrangian} submanifold is a real $n$-dimensional submanifold such that
		\begin{equation}\label{S_6_Eqn_Lagangle}
			\omega|_{L}=0, \IIm{e^{\II\hat{\theta}}\Omega|_{L}}=0, \hat{\theta}\in [0,2\pi).
		\end{equation}
	We call a special Lagrangian \textbf{exact} if $\lambda|_{L}$ is an exact form. In this section, we will discuss the relationship between the multi-valued graph $\ud f_{\bs{h}}$ and a family of exact special Lagrangians, which are known as Lawlor's necks \cite{Lawlor89, Harvay90}.\\
	
	Let $\Pi_{\bs{0}}, \Pi_{\bs{\phi}}$ be two special Lagrangian planes defined by
		\begin{equation*}
		\begin{split}
			&\Pi_{\bs{0}}=\{(x_{1},\cdots, x_{n})|x_{i}\in \R \},\\
			&\Pi_{\bs{\phi}}=\{(e^{\II \phi_{1}}x_{1},\cdots, e^{\II \phi_{n}}x_{n})|x_{i}\in \R\},\phi_{i}\in (0,\pi)
		\end{split}
		\end{equation*}
	The condition that Lagrangian angles $\hat{\theta}$ on $\Pi_{0},\Pi_{\bs{\phi}}$ as in Equation \ref{S_6_Eqn_Lagangle} are the same is equivalent to	
		\begin{equation*}
			\sum_{i=1}^{n}\phi_{i}=m\pi, m=1,\cdots,n-1.
		\end{equation*}
	A Lawlor's neck is a smoothing of the singular special Lagrangian $\Pi_{\bs{0}}\cup\Pi_{\bs{\phi}}$ in the case when
		\begin{equation}\label{S_6_Eqn_Angle}
			\sum_{i=1}^{n}\phi_{i}=\pi.
		\end{equation}
	For our purposes, we will consider a model acted by
		\begin{equation*}
			\textrm{diag}(e^{-\II \phi_{1}/2},\cdots,e^{-\II\phi_{n-1}/2},e^{\II(\pi-\phi_{n})/2})\in SU(n).
		\end{equation*}
	Denote two special Lagrangian planes as
		\begin{equation*}
		\begin{split}
			&\Pi_{-}=\{(e^{-\II \phi_{1}/2}x_{1},\cdots,e^{-\II \phi_{n-1}/2}x_{n-1},e^{\II(\pi-\phi_{n})/2}x_{n})\};\\
			&\Pi_{+}=\{(e^{\II \phi_{1}/2}x_{1},\cdots,e^{\II \phi_{n-1}/2}x_{n-1},e^{\II(\pi+\phi_{n})/2}x_{n})\};
		\end{split}
		\end{equation*}
	The Lawlor's necks $L_{\bs{c}}$ are constructed as follows. Let
		\begin{equation*}
			\bs{c}=(c_{1},\cdots, c_{n}), c_{i}>0,
		\end{equation*}
	be a multiple of positive parameters. Define
		\begin{equation*}
			P(y^{2})=\big(\prod_{k=1}^{n}(1+c_{k}^{2}y^{2})-1\big)/y^{2}
		\end{equation*}
		and
		\begin{equation*}
		\begin{split}
			&\psi_{i}(s)=\int_{0}^{s}\frac{c_{i}^{2}\ud y}{(1+c_{i}^{2}y^{2})\sqrt{P(y^{2})}},\\
			&\phi_{i}=\int_{-\infty}^{\infty}\frac{c_{i}^{2}\ud y}{(1+c_{i}^{2}y^{2})\sqrt{P(y^{2})}}.
		\end{split}
		\end{equation*}
	Here, $\phi_{i}$ satisfies the relation in Equation \ref{S_6_Eqn_Angle}. Let
		\begin{equation*}
			\begin{split}
				&z_{k}(s)=e^{\II \psi_{k}(s)}\sqrt{c_{k}^{-2}+s^{2}}, 1\leq k\leq n-1;\\
				&z_{n}(s)=e^{\II (\pi/2+\psi_{k}(s))}\sqrt{c_{n}^{-2}+s^{2}}.
			\end{split}
		\end{equation*}
		\begin{definition}\label{S_6_Defn_Lawlor}
			We define Lawlor's necks $L_{\bs{c}}$ as
			\begin{equation*}
				L_{\bs{c}}=\{(z_{1}(s)w_{1},\cdots, z_{n}(s)w_{n})|(w_{1},\cdots,w_{n})\in S^{n-1},s\in \R\}.
			\end{equation*}
		\end{definition}
	
	It can be checked that $L_{\bs{c}}$ is asymptotic to $\Pi_{+}\cup \Pi_{-}$ at infinity. Moreover, the projection $Pr_{X}: L_{\bs{c}}\to \Pi_{0}$ is given by 
		\begin{equation*}
			(w_{i},s)\mapsto (\cos(\psi_{i}(s))\sqrt{c_{i}^{-2}+s^{2}}w_{i},-\sin(\psi_{n}(s))\sqrt{c_{n}^{-2}+s^{2}}w_{n}).
		\end{equation*}
	Define 
		\begin{equation}\label{S_6_Eqn_ProjE}
			\begin{split}
				&m_{i}(s)=\cos(\psi_{i}(s))\sqrt{c_{i}^{-2}+s^{2}};\\
				&m_{n}(s)=\sin(\psi_{n}(s))\sqrt{c_{n}^{-2}+s^{2}},
			\end{split}
		\end{equation}
	accordingly. It can be checked that $m_{k}'(s)>0, 1\leq k \leq n$. As a consequence, the projection is a $2$-$1$ map away from a codimension-$2$ ellipsoid
		\begin{equation*}
			L_{\bs{c}}\cap \Pi_{0}:\sum_{k=1}^{n-1}c_{k}^{2}x_{k}^{2}=1, x_{n}=0.
		\end{equation*}
	In other words, $L_{\bs{c}}$ can be viewed as a multivalued graph from $\Pi_{\bs{0}}=\R^{n}$ to $\CC^{n}=T^{*}\R^{n}$. In particular, if we view $-\lambda_{L}$ as a multivalued one form on $\Pi_{0}$, then the multivalued graph is defined tautologically by $-\lambda_{L}$. Moreover, since $L_{\bs{c}}$ is diffeomorphic to $S^{n-1}\times \R$, which is simply connected. The Liouville form is exact
		\begin{equation*}
			\lambda_{L}=-\ud F_{\bs{c}}, F_{\bs{c}}\in C^{\infty}(L_{\bs{c}}),
		\end{equation*}
		on $L_{\bs{c}}$. We can normalize $F_{\bs{c}}$ so that $F_{\bs{c}}\equiv 0$ on the ellipsoid, for $\lambda_{\bs{c}}$ vanishes on the ellipsoid $L_{\bs{c}}\cap \Pi_{\bs{0}}$. In this case, $F_{\bs{c}}$ is a $\Z_{2}$-function on $\Pi_{\bs{0}}$.\\
	
	Now consider the case when $\phi_{i}, 1\leq i\leq n-1$ is sufficiently small. Define a family of multiple, parameterized by a sufficiently large positive number $t$
		\begin{equation*}
			\bs{c}_{t}:=(c_{1},\cdots,c_{n-1}, t).
		\end{equation*}
	Denote $C(y)$ as $\prod_{k=1}^{n-1}(1+c_{k}^{2}y)$. Computation implies
		\begin{equation}\label{S_6_Eqn_Angle_1}
		\begin{split}
			&\phi_{i}=t^{-1}\int_{-\infty}^{\infty}\frac{c_{i}^{2}\ud y}{(1+c_{i}^{2}y^{2})\sqrt{C(y^{2})}}+O(t^{-3}), 1\leq i\leq n-1\\
			&\phi_{n}=\frac{\pi}{2}-t^{-1}\int_{-\infty}^{\infty}\frac{C'(y^{2})\ud y}{C(y^{2})^{3/2}}+O(t^{-3})
		\end{split}
		\end{equation}
	
	On the other hand, the multivalued graph of $t^{-1}\ud f_{\bs{h}}$ is asymptotic to $\tilde{\Pi}_{+}\cup \tilde{\Pi}_{-}$, where
		\begin{equation*}
		\begin{split}
			&\tilde{\Pi}_{-}=\{(e^{-\II \tilde{\phi}_{1}/2}x_{1},\cdots,e^{-\II \tilde{\phi}_{n-1}/2}x_{n-1},e^{\II(\pi-\tilde{\phi}_{n})/2}x_{n})\};\\
			&\tilde{\Pi}_{+}=\{(e^{\II \tilde{\phi_{1}}/2}x_{1},\cdots,e^{\II \tilde{\phi}_{n-1}/2}x_{n-1},e^{\II(\pi+\tilde{\phi}_{n})/2}x_{n})\};\\
			&\tilde{\phi}_{i}=2\arctan(2t^{-1}a_{i})\in (0,\pi), 1\leq i\leq n;\\
			&\tilde{\phi}_{n}=\pi+2\arctan(2t^{-1}a_{n}) \in (0,\pi).
		\end{split}
		\end{equation*}
	The angles $\tilde{\phi}_{i}$ do not obey the equation \ref{S_6_Eqn_Angle}. Consequently, there is no uniform choice of the Lagrangian angle $\hat{\theta}$ that simultaneously makes $\tilde{\Pi}_{\pm}$ into special Lagrangians. But as $t\to +\infty$, they satisfy
		\begin{equation}\label{S_6_Eqn_AngleIn}
			\pi<\sum_{i=1}^{n}\tilde{\phi}_{i}<\pi+O(t^{-3}).
		\end{equation}  
    Indeed, use Taylor's expansion
		\begin{equation}\label{S_6_Eqn_Angle_2}
			\begin{split}
				&\tilde{\phi}_{i}=t^{-1}\int_{-\infty}^{\infty}\frac{\sqrt{\sigma_{n-1}(\bs{h})}\ud u}{(u^{2}+h_{i}^{2})\sqrt{S(u^{2})}}+O(t^{-3});\\
				&\tilde{\phi}_{n}=\pi-t^{-1}\int_{-\infty}^{\infty}\frac{\sqrt{\sigma_{n-1}(\bs{h})}S'(u^{2})\ud u}{S(u^{2})^{3/2}}+O(t^{-3}).
			\end{split}
		\end{equation}
	
	If we let $c_{i}=\frac{1}{h_{i}}$ and compare Equation \ref{S_6_Eqn_Angle_1} and \ref{S_6_Eqn_Angle_2}, the difference between $\phi_{i}$ and $\tilde{\phi}_{i}$ is bounded by $O(t^{-3})$. Moreover, we will have the following.
		\begin{proposition}\label{S_6_Prop_Converge}
			Suppose the multiple is given by
				\begin{equation*}
					\bs{c}_{t}=(\frac{1}{h_{1}},\cdots,\frac{1}{h_{n-1}},t),
				\end{equation*}
				then the $2$-valued function $tF_{\bs{c}_{t}}$ converges to the $\Z_{2}$-harmonic function $f_{\bs{h}}$ in $C^{\infty}_{loc}(\R^{n}\setminus E_{\bs{h}})$.
		\end{proposition}
		\begin{proof}
			Let $U\Subset \R^{n}\setminus E_{\bs{h}}$ be a simply connected open set, we can choose a single valued branch of $F_{\bs{c}_{t}}$. It follows from the fact that $\ud F_{\bs{c}_{t}}$ defines a special Lagrangian, $tF_{\bs{c}_{t}}$ is a solution to
				\begin{equation*}
					\IIm{\det(I+\sqrt{-1}t^{-1}Hess(F))}=0, \text{ on $U$.}
				\end{equation*}
			
			Let $\sigma_{k}(Hess(f))$ be the $k$-th symmetric polynomial of the eigenvalues of $Hess(F)$. Expand the above equation
				\begin{equation*}
					\Delta F-t^{-2}\sigma_{3}(Hess(F))+\cdots+(-1)^{k}t^{-2k}\sigma_{2k+1}(Hess(F)),
				\end{equation*}
				where $k=\lfloor \frac{n-1}{2}\rfloor$. We will show that $tF_{\bs{c}_{t}}$ converges to a smooth function $F$ on $U$, and the equation will imply $F$ is harmonic. The uniqueness theorem will imply $F=f_{\bs{h}}$. Consequently, $tF_{\bs{c}_{t}}$ converges to $f_{\bs{h}}$ in $C_{loc}^{\infty}(\R^{n})$.\\
				
			To see this, we need to understand the behavior of the projection map when $t$ is large. Recall that
				\begin{equation*}
					Pr_{X}=(m_{i}(s)w_{i},-m_{n}(s)w_{n}),
				\end{equation*}
			where
				\begin{equation*}
				\begin{split}
					&m_{i}(s)=\sqrt{c_{i}^{-2}+s^{2}}\cos(\psi_{i}(s)), 1\leq i\leq n,\\
					&m_{n}(s)=\sqrt{t^{-2}+s^{2}}\sin(\psi_{n}(s)).
				\end{split}
				\end{equation*}
			Similarly, we define $Pr_{Y}$ as the projection of the $2$-valued graph onto the cotangent direction
				\begin{equation*}
					Pr_{Y}=(d_{i}(s)w_{i}, d_{n}(s)),
				\end{equation*}
			where
				\begin{equation*}
					\begin{split}
						& d_{i}(s)=t\sqrt{c_{i}^{-2}+s^{2}}\sin(\psi_{i}(s)), 1\leq i\leq n-1,\\
						& d_{n}(s)=\sqrt{1+s^{2}t^{2}}\cos(\psi_{n}(s)).
					\end{split}
				\end{equation*}
			
			We now consider the asymptotic behavior of $\psi_{i}(s), \psi_{n}(s)$ as $t \to \infty$. Let $C(x)=\prod_{j=1}^{n-1}(1+c_{j}^{2}x)$ and $\tilde{C}(x)=(C(x)-1)/x$. Then
				\begin{equation*}
					t\psi_{i}(s)=\int_{0}^{s}\frac{c_{i}^{2}\ud u}{(1+c_{i}^{2}u^{2})\sqrt{C(u^{2})+t^{-2}\tilde{C}(u^{2})}}, 1\leq i\leq n-1
				\end{equation*}
				the function $t\psi_{i}(s)$ converges to $\beta_{i}(s)$ in $C_{loc}^{\infty}(\R)$ with error bounded by $O(t^{-2})$. Here,
					\begin{equation}\label{S_6_Eqn_Est_1}
						\beta_{i}(s)=\int_{0}^{s}\frac{c_{i}^{2}\ud u}{(1+c_{i}^{2}u^{2})\sqrt{C(u^{2})}}
					\end{equation}
				Similarly, one can show that $t(\psi_{n}(s)-\arctan(st))$ converges to
					\begin{equation}\label{S_6_Eqn_Est_2}
						\beta_{n}(s):=-\int_{0}^{s}\frac{\tilde{C}(u^{2})\ud u}{C(u^{2})+\sqrt{C(u^{2})}}
					\end{equation}
					in $C_{loc}^{\infty}(\R)$ with an error bounded by $O(t^{-2})$. To understand the behavior of $\psi_{i}, \psi_{n}$, we will consider two cases. The first one, $|s|$, is bounded from below and above, and the second one, $|s|$, is sufficiently small.
				\begin{enumerate}
					\item Let $0<v_{2}<v_{1}$, in the region $\{v_{2}/2<|s|<2v_{1}\}$
						\begin{equation*}
							\psi_{i},\psi_{i}',\psi_{i}'',\pi/2-\psi_{n},\psi_{n}',\psi_{n}''=O(t^{-1})
						\end{equation*}

					\item In the region $\{|s|<\delta_{1}<\delta\ll 1\}$
						\begin{equation*}
							\begin{split}
								&\psi_{i},\psi_{i}',\psi_{i}''=O(t^{-1}),\\
								& \psi_{n}=\arctan(st)+O(st^{-1}).
							\end{split}
						\end{equation*}
						In particular, we have the following estimates
						\begin{equation*}
							\begin{split}
								& \cos(\psi_{n})=\frac{1}{\sqrt{1+s^{2}t^{2}}}-\frac{s\beta_{n}(s)}{\sqrt{1+s^{2}t^{2}}}+O(t^{-3}),\\
								& \sin(\psi_{n})=\frac{st}{\sqrt{1+s^{2}t^{2}}}+\frac{t^{-1}\beta_{n}(s)}{\sqrt{1+s^{2}t^{2}}}+O(t^{-3}).
							\end{split}
						\end{equation*}
						
				\end{enumerate}
			Therefore, $Pr_{X},Pr_{Y}$ converges to
				\begin{equation}\label{S_6_Eqn_ApproxPr}
					\begin{split}
						&Pr_{X}^{\infty}:=(w_{i}\sqrt{c_{i}^{-2}+s^{2}},-w_{n}s),\\
						&Pr_{Y}^{\infty}:=(w_{i}\beta_{i}(s)\sqrt{c_{i}^{-2}+s^{2}}, w_{n}(1-s\beta_{n}(s)) ),
					\end{split}
				\end{equation}
			in $C^{\infty}_{loc}(\R\times S^{n-1})$ respectively, with an error bounded by $O(t^{-2})$.\\
			
			We define two kinds of open sets in $\R^{n}\setminus E_{\bs{h}}$, which exhaust $\R^{n}\setminus E_{\bs{h}}$
				\begin{equation*}
					U_{v_{1},v_{2}}:=U_{v_{1}}\setminus \overline{U_{v_{2}}},
				\end{equation*}
				where
				\begin{equation*}
					U_{v}:\frac{x_{1}^{2}}{c_{1}^{-2}+v^{2}}+\cdots+\frac{x_{n-1}^{2}}{c_{n-1}^{-2}+v^{2}}+\frac{x_{n}^{2}}{v^{2}}< 1,
				\end{equation*}
				and
				\begin{equation*}
					\begin{split}
				    &V_{\delta_{1},\delta_{2}}:c_{1}^{2}x_{1}^{2}+\cdots+c_{n-1}^{2}x_{n-1}^{2}<1-\delta_{1}^{2}, |x_{n}|<\delta_{2},\\
						&V_{0}:c_{1}^{2}x_{1}^{2}+\cdots+c_{n-1}^{2}x_{n-1}^{2}<1, x_{n}=0.
					\end{split}
				\end{equation*}

			When $v_{1},v_{2}$ are fixed and $t$ is sufficiently large, then
				\begin{equation}\label{S_6_Eqn_Approx_1}
					Pr_{X}=(w_{i}\sqrt{c_{i}^{-2}+s^{2}},-w_{n}s)+O(t^{-2}),
				\end{equation}
			in the region $\{v_{2}/2<|s|<2v_{1}\}$. Consequently, when $t$ is large
				\begin{equation*}
					U_{v_{1},v_{2}}\subset Pr_{X}(\{v_{2}/2<|s|<2v_{1}\}).
				\end{equation*}

			Meanwhile, let $\delta$ be a sufficiently small constant, in the region $\{|s|<\delta_{1}<\delta\ll 1\}$, the Equation \ref{S_6_Eqn_Est_2} implies
				\begin{equation*}
					Pr_{X}=(w_{i}\sqrt{c_{i}^{-2}+s^{2}},-w_{n}s)+O(t^{-1}),
				\end{equation*}
			and
				\begin{equation}\label{S_6_Eqn_Approxn}
					\big|c_{1}^{2}x_{1}^{2}+\cdots+c_{n-1}^{2}x_{n-1}^{2}-(1-w_{n}^{2})\big|<\tilde{C}^{2}s^{2},
				\end{equation}
			for a constant $\tilde{C}$ depending only on $\delta, c_{1},\cdots,c_{n-1}$. When $t>N(\delta_{1})$ is sufficiently large, we conclude that
				\begin{equation*}
					V_{2\tilde{C}\delta_{1},\frac{1}{2}\delta_{1}}\subset Pr_{X}(\{|s|<\delta_{1},w_{n}>2\tilde{C}\delta_{1}\}).
				\end{equation*}
			
			It follows naturally that:
				\begin{equation*}
					\bigcup_{0<v_{2}<v_{1}}U_{v_{1},v_{2}}=\R^{n}\setminus \overline{V_{0}},
				\end{equation*}
			and that
				\begin{equation*}
					\bigcup_{0<\delta_{1}<\delta}V_{2\tilde{C}\delta_{1},\delta_{1}}\supset V_{0}.
				\end{equation*}
			
			To estimate the $\|tF_{\bs{c}_{t}}\|$, we only need to give an upper bound on the Jacobian $DPr_{X}^{-1}$ away from the branching set. When $v>|s|>\delta>0$, $|(DPr_{X}^{\infty})^{-1}|$ is bounded from above. It suffices to bound $|(DPr_{X}^{\infty})^{-1}|$ in $\{|w_{n}|>2\tilde{C}\delta_{1}\}$. In this case, choose
				\begin{equation*}
					(w_{1},\cdots, w_{n-1},\sqrt{1-w_{1}^{2}-\cdots-w_{n-1}^{2}}),
				\end{equation*}
				as local coordinates of $S^{n-1}$. The determinant
				\begin{equation*}
					\det(DPr_{X}^{\infty})=-w_{n}\prod_{j}^{n-1}(c_{j}^{-2}+s^{2})^{1/2}+O(\frac{s^{2}}{|w_{n}|}).
				\end{equation*}
				Consequently, $|(DPr_{X}^{\infty})^{-1}|$ is bounded in $V_{2\tilde{C}\delta_{1},\frac{1}{2}\delta_{1}}$. Notice that
				\begin{equation*}
					Hess(tF_{\bs{c}_{t}})=DPr_{Y}\circ (DPr_{X})^{-1},
				\end{equation*}
				we conclude $|\nabla^{k}tF_{\bs{c}_{t}}|<C_{k}$ on $V_{2\tilde{C}\delta_{1},\delta_{1}}$. Moreover, at the branching ellipsoid $w_{n}=s=0$, one can check that
					\begin{equation*}
						Pr_{X}^{\infty}=(\frac{w_{i}}{c_{i}^{2}\sqrt{1-w_{n}^{2}}}(1+\frac{c_{i}^{2}s^{2}-w_{n}^{2}}{2}),-w_{n}s)+O((s^{2}+w_{n}^{2})^{2})
					\end{equation*}
					and we can show
					\begin{equation*}
						|(DPr_{X}^{\infty})^{-1}|=O(r^{-1/2}),
					\end{equation*}
					where $r$ is the distance to the branching set. Combine with the normalization condition $tF_{\bs{c}_{t}}=0$ on the branching set, we conclude that $tF_{\bs{c}_{t}}$ converges to a $2$-valued smooth function $F$ in $C^{\infty}_{loc}(\R^{n}\setminus E_{\bs{h}})$.\\
					
			As we have shown before, $F$ is $\Z_{2}$-harmonic on $\R^{n}\setminus E_{\bs{h}}$. In particular, $|\nabla F|, |F|$ are bounded near $E_{\bs{h}}$. We can extend $|F|$ to $E_{\bs{h}}$ continuously by setting $|F|=0$ on $E_{\bs{h}}$. It remains to show $F=f_{\bs{h}}$. Because $\ud F$ and $\ud f_{\bs{h}}$ have the same asymptotic behavior at infinity, we obtain
				\begin{equation*}
					F-f_{\bs{h}}=C+O(|\bs{x}|^{2-n}).
				\end{equation*}
				The uniqueness Theorem \ref{S_5_Prop_Uniqueness} and Equation \ref{S_3_Eqn_EllHar} imply
				\begin{equation*}
					F=f_{\bs{h}}+c\int_{0}^{\mu_{n}}\frac{\ud u}{\sqrt{S(u^{2})}},
				\end{equation*}
				for a constant $c$. Notice that
				\begin{equation*}
					\ud (\int_{0}^{\mu_{n}}\frac{\ud u}{\sqrt{S(u^{2})}})\sim A(\Theta)\RRe{z^{-1/2}},
				\end{equation*}
				near $E_{\bs{h}}$, which is unbounded. This forces $c=0$, since $|\ud F|$ and $|\ud f_{\bs{h}}|$ are bounded near $E_{\bs{h}}$.
		\end{proof}
		
	As a corollary of the above proposition, we are able to show	
		\begin{corollary}\label{S_6_Cor_nonVanishing}
			The $\Z_{2}$-harmonic one-form $\ud f_{\bs{h}}\neq 0$ on $\R^{n}\setminus E_{\bs{h}}$. 
		\end{corollary}
		\begin{proof}
			Use the above proposition, the graph of $\ud f_{\bs{h}}$ is parametrized by
				\begin{equation*}
					(Pr_{X}^{\infty},Pr_{Y}^{\infty}).
				\end{equation*}
				It is easy to check that $Pr_{Y}^{\infty}=0$ if and only if $s=w_{n}=0$, which corresponds to the branching set $E_{\bs{h}}$.
		\end{proof}
	
\nocite{*}

\bibliographystyle{alpha}

\end{document}